\documentclass[reqno,12pt]{amsart}

\usepackage{amsmath}
\usepackage{amsfonts}
\usepackage{amssymb}
\usepackage{eufrak}
\usepackage{amscd}
\usepackage{amsthm}
\usepackage{amstext}
\usepackage[all]{xy}
   \theoremstyle{plain}%default
   \newtheorem{thm}{Theorem}%[section]
   
   \newtheorem{lem}[thm]{Lemma}
   
   \theoremstyle{definition}
   
   \newtheorem{defn}[thm]{Definition}
   
   \theoremstyle{remark}
   
   \newtheorem{remark}[thm]{Remark}

\author{V. Manuilov}

\date{}

\address{Moscow State University,
Leninskie Gory, Moscow, 
119991, Russia}

\email{manuilov@mech.math.msu.su}

\thanks{The author acknowledges partial support by the RFBR grant No. 14-01-00007.}

\title{A finitedimensional version of Fredholm representations}

\begin{document}
\maketitle

\begin{abstract}
We consider pairs of maps from a discrete group $\Gamma$ to the unitary group. The deficiencies of these maps from being homomorphisms may be great, but if they are close to each other then we call such pairs {\it balanced}. We show that balanced pairs determine elements in the $K^0$ group of the classifying space of the group. We also show that a Fredholm representation of $\Gamma$ determines balanced pairs. 

\end{abstract}

\section{Introduction}

It is well known that various generalizations of unitary group representations, e.g. almost representations, Fredholm representations, representations into $U(p,q)$, quasi-representations etc. can be viewed as representatives of the $K$-homology of the group $C^*$-algebra. It is interesting to know, how far we can generalize the notion of a representation (or a pair of representations) keeping the property to determine a class in $K$-homology. It was shown recently
in \cite{Man-K} that K-theory elements can be represented not necessarily by pairs of projections,
but by pairs satisfying weaker properties. We follow this to find a generalization for
pairs of matrix-valued functions on a group called balanced pairs.

Let $\Gamma$ be a finitely generated group, and let $M_n$ denote the $C^*$-algebra of operators on the $n$-dimensional Hilbert space.
Given a map $\pi:\Gamma\to M_n$, and $g,h\in\Gamma$, denote by $M_\pi(g,h)=\pi(gh)-\pi(g)\pi(h)$ the defect, i.e. the deviation of $\pi$ from multiplicativity.  
 
\begin{defn}
Given a finite set $F\in\Gamma$ and $\varepsilon>0$, a pair $(\pi^+,\pi^-)$ of maps $\Gamma\to M_n$ satisfying $\pi^\pm(g^{-1})=\pi^\pm(g)^*$ for any $g\in F$, is {\it $(F,\varepsilon)$-admissible} if 
\begin{equation}\label{condition1}
\|M_{\pi^\pm}(g,h)(\pi^+(\gamma)-\pi^-(\gamma)\|<\varepsilon
\end{equation}
for any $g,h,\gamma\in F$.
A pair $(\pi^+,\pi^-)$ satisfying $\pi^\pm(g^{-1})=\pi^\pm(g)^*$ for any $g\in F$, is {\it $(F,\varepsilon)$-balanced} if   
\begin{equation}\label{condition2}
\|M_{\pi^+}(g,h)-M_{\pi^-}(g,h)\|<\varepsilon
\end{equation}
for any $g,h\in F$, and
\begin{equation}\label{condition2'}
\|\pi^+(k)M_{\pi^+}(g,h)-\pi^-(k)M_{\pi^-}(g,h)\|<\varepsilon
\end{equation}
for any $g,h,k\in F$.

\end{defn}

A family of pairs of maps $(\pi^\pm_n)_{n\in\mathbb N}:\Gamma\to M_{k_n}$ is {\it asymptotically admissible} (resp., {\it asymptotically balanced}) if, for any finite $F\subset\Gamma$, the pair $(\pi^+_n,\pi^-_n)$ is $(F,\varepsilon_n)$-admissible (resp., balanced) with $\varepsilon_n\to 0$ as $n\to\infty$. We also use these terms for families of maps with continuous parameter $t\in[0,\infty)$.

A similar definition appeared in \cite{MY IJM}, but there we required that the range of the maps $\pi^\pm$ lies in the unitary group of $M_n$. Here we don't assume $\pi^\pm(g)$ to be invertible.

We shall show that asymptotically admissible (resp., balanced) pairs can be viewed as finitedimensional versions of Fredholm representations.

\section{Making Fredholm representations finitedimensional}

The definition of Fredholm representations is due to A.S.Mishchenko \cite{Mish}. A Fredholm representation is a triple $(\pi^+,F,\pi^-)$ of two representations, $\pi^+,\pi^-$, of $\Gamma$, on a Hilbert space $H$, and of a Fredholm `intertwining' operator $F$ on $H$, such that $\pi^+(g)F-F\pi^-(g)$ is compact for any $g\in\Gamma$. Well known simplifications allow to drop out either one of the two representations, or $F$. For example, we may change $F$ by either an isometry or a coisometry, and then change $\pi^-$ by $F\pi^- F^*$ or $\pi^+$ by $F^*\pi^+ F$, which gives us a triple of the form $(\pi^+,\operatorname{Id},\pi^-)$. So, from now on, let us consider pairs $(\pi^+,\pi^-)$ with $\pi^+(g)-\pi^-(g)$ compact for any $g\in\Gamma$ as Fredholm representations.

Let $L_n\in\mathbb B(H)$ be an increasing sequence of subspaces, such that $\dim L_n=n$ and $\cup_{n\in\mathbb N}L_n$ is dense in $H$, and let $P_n$ denote the projection onto $L_n$. For any operator $A$, set $A_n=P_nA|_{L_n}$. Similarly, we write $\pi_n$ for the map given by $g\mapsto(\pi(g))_n$.

\begin{thm}
For any Fredholm representation $(\pi^+,\pi^-)$, the sequence $(\pi^+_n,\pi^-_n)$ is asymptotically admissible and asymptotically balanced.

\end{thm}
\begin{proof}
One obviously has $\pi_n^\pm(g^{-1})=\pi_n^\pm(g)^*$ for any $g\in\Gamma$. Let us check (\ref{condition1}).
Since $\pi^+(\gamma)-\pi^-(\gamma)$ is compact for any $\gamma\in \Gamma$, so for any $\varepsilon>0$ and for any finite $F\subset\Gamma$, there is $K$ such that for any $n>k>K$ one has 
$$
\|\pi^+_n(\gamma)-\pi^-_n(\gamma)-P_k(\pi^+_n(\gamma)-\pi^-_n(\gamma))P_k\|<\varepsilon
$$ 
for any $\gamma\in F$. Fix $k>K$, then it suffices to show that for sufficiently large $n$, $\|P_k(\pi^\pm_n(gh)-\pi^\pm_n(g)\pi^\pm_n(h))\|$ and $\|(\pi^\pm_n(gh)-\pi^\pm_n(g)\pi^\pm_n(h))P_k\|$ can be made smaller than $\varepsilon$.
As the two terms are similar, let us estimate the first one.

We have 
$$
\|P_k(\pi^\pm_n(gh)-\pi^\pm_n(g)\pi^\pm_n(h))\|=\|P_kP_n\pi^\pm(g)(1-P_n)\pi^\pm(h)P_n\|
$$
$$
=\|P_k\pi^\pm(g)(1-P_n)\pi^\pm(h)P_n\|\leq\|P_k\pi^\pm(g)(1-P_n)\|.
$$
Then for a finite number of elements $g\in F\subset \Gamma$ and for the fixed $k$ it is always possible to find $N$ such that for any $n>N$ one has $\|P_k\pi^\pm(g)(1-P_n)\|<\varepsilon$.

Now let us check (\ref{condition2}). We use the notation from the previous paragraph, namely the projections $P_k$ and $P_n$. Decompose the Hilbert space $H$ as $H=P_kH\oplus(P_n-P_k)H\oplus(1-P_n)H$, and write operators on $H$ as $3{\times}3$ matrices with respect to this decomposition. 

Let $\pi^\pm(g)=A^\pm=(a_{ij}^\pm)_{i,j=1}^3$, $\pi^\pm(h)=B^\pm=(b_{ij}^\pm)_{i,j=1}^3$, $\pi^\pm(gh)=C^\pm=(c_{ij}^\pm)_{i,j=1}^3$. We know that 
\begin{equation}\label{exact1}
A^\pm B^\pm=C^\pm, 
\end{equation}
and that $\|x_{ij}^+-x_{ij}^-\|<\varepsilon$, where $x=a,b,c$, for all $i,j=1,2,3$ except the case $i=j=1$. 
Using (\ref{exact1}), we obtain that
$$
M_{\pi^\pm_n}(g,h)=
\left(\begin{matrix}c^\pm_{11}&c^\pm_{12}\\c^\pm_{21}&c^\pm_{22}\end{matrix}\right)-
\left(\begin{matrix}a^\pm_{11}&a^\pm_{12}\\a^\pm_{21}&a^\pm_{22}\end{matrix}\right)
\left(\begin{matrix}b^\pm_{11}&b^\pm_{12}\\b^\pm_{21}&b^\pm_{22}\end{matrix}\right)
=
\left(\begin{matrix}a^\pm_{13}b^\pm_{31}&a^\pm_{13}b^\pm_{32}\\
a^\pm_{23}b^\pm_{21}&a^\pm_{23}b^\pm_{32}\end{matrix}\right),
$$
hence $\|M_{\pi^+_n}(g,h)-M_{\pi^-_n}(g,h)\|\leq 
4\max_{(i,j),(k,l)\neq(1,1)}\|a^+_{ij}b^+_{kl}-a^-_{ij}b^-_{kl}\|<8\varepsilon$.

Finally, to obtain (\ref{condition2'}), we have to combine (\ref{condition1}) and (\ref{condition2}).

\end{proof}

One can replace the discrete parameter by a continuous one. This follows from the following Lemma.

\begin{lem}
For $t\in[n,n+1]$ set $\pi^\pm_t(g)=t\pi^\pm_n(g)+(1-t)\pi^\pm_{n+1}(g)$. Then the family of pairs $(\pi^+_t,\pi^-_t)$ is asymptotically admissible and asymptotically balanced.

\end{lem}
\begin{proof}
Direct calculation similar to that above. 

\end{proof}

Let $X$ be a compact Hausdorff space, $\pi_1(X)=\Gamma$, $\{U_i\}_{i\in I}$ a finite covering, and let $\varphi_i$, $i\in I$, be continuous functions on $X$ such that $0\leq \varphi_i(x)\leq 1$, $i\in I$, $x\in X$, $\operatorname{supp}\varphi_i\subset U_i$ and $\sum_{i\in I}\varphi^2_i(x)=1$ for any $x\in X$. Let $\gamma=\{\gamma_{ij}\}_{i,j\in I}$ be a $\Gamma$-valued cocycle, i.e. $\gamma_{ji}=\gamma_{ij}^{-1}$ for any $i,j\in I$, and $\gamma_{ij}\in\Gamma$ and $\gamma_{ij}\gamma_{jk}=\gamma_{ik}$ whenever $U_i\cap U_j\cap U_k$ is not empty.

Then 
\begin{equation}\label{proj}
p(x)=(p_{ij}(x))_{i,j\in I},\mbox{\ where\ } p_{ij}(x)=\varphi_i(x)\varphi_j(x)\gamma_{ij}, 
\end{equation}
is known to be idempotent for each $x\in X$.

For a map $\pi:\Gamma\to M_n$, put 
\begin{equation}\label{A}
A_\pi(x)=\pi(p(x))=(\varphi_i(x)\varphi_j(x)\pi(\gamma_{ij}))_{i,j\in I}. 
\end{equation}
When $\pi$ is a (unitary) group representation then $A_\pi$ is a (selfadjoint) projection. 

For shortness' sake set $A_{\pi^+}=a$, $A_{\pi^-}=b$. Let $\delta=|I|\cdot\varepsilon$.

If $(\pi^+,\pi^-)$ is $(F,\varepsilon)$-admissible then the pair $(a,b)$ satisfies the following conditions:
\begin{equation}\label{selfadjoint}
a^*=a;\quad b^*=b; 
\end{equation}
\begin{equation}\label{K1}
\|(a^2-a)(a-b)\|<\delta; \quad \|(b^2-b)(a-b)\|<\delta.
\end{equation}

If $(\pi^+,\pi^-)$ is $(F,\varepsilon)$-balanced then the pair $(a,b)$ satisfies (\ref{selfadjoint}) and
\begin{equation}\label{K2}
\|f(a)-f(b)\|<\delta
\end{equation}
for $f(t)=t(1-t)$ and $f(t)=t^2(1-t)$.

\section{Relation to $K$-theory}

Consider the following two sets of relations on selfadjoints $a$ and $b$:

\begin{equation}\label{K-1}
0\leq a,b\leq 1;\quad (a^2-a)(a-b)=(b^2-b)(a-b)=0;
\end{equation}
and 
\begin{equation}\label{K-2}
p(a)=p(b) \mbox{\ \ for\ } p(t)=t(1-t) \mbox{\ and\ } p(t)=t^2(1-t).
\end{equation}

In \cite{Man-K} it was shown that the $K_0$ group of a $C^*$-algebra $A$ is the set of homotopy classes of selfadjoint pairs $(a,b)$ of matrices over $A$ satisfying either (\ref{K-1}) or (\ref{K-2}). 

As we may replace exact projections by almost projections (i.e. selfadjoints $A$ with $\|A^2-A\|<\frac{1}{4}$), so the relations (\ref{K-1}) and (\ref{K-2}) can be replced by their approximate versions: (\ref{K1}) plus $0\leq a,b\leq 1$, and (\ref{K2}), respectively, for sufficiently small $\delta$. It was shown in \cite{Man-K} in particular, that the element of the $K_0$ group corresponding to a pair $(a,b)$ satisfying (\ref{K-1}) is given by the formal difference $[P]-[Q]$, where 
$$
P=P(a,b)=\left(\begin{matrix}1-b&g(a)\\g(a)&a\end{matrix}\right)
$$ 
is a projection ($g(t)=\sqrt{t-t^2}$) and $Q=\left(\begin{matrix}1&0\\0&0\end{matrix}\right)$. If $a$ and $b$ satisfy (\ref{K1}) and $0\leq a,b\leq 1$ then $P$ is only an almost projection, but $[P]-[Q]$ still determines an element in $K_0$.

Note that if $a$, $b$ are genuine projections then $P=P(a,b)=\left(\begin{matrix}1-b&0\\0&a\end{matrix}\right)$, hence $[P]-[Q]=[1-b]+[a]-[1]=[a]-[b]$.

No explicit formula was given in \cite{Man-K} for pairs satisfying (\ref{K-2}). Since $A_{\pi^\pm}$ do not need to satisfy $0\leq A_{\pi^\pm}\leq 1$, there are two ways to proceed. We may apply the cutting function $h$, $h(t)=\left\lbrace\begin{array}{rl}0,&t<0;\\t,&0\leq t\leq 1;\\1,&t>1\end{array}\right.$ to make $h(A_{\pi^\pm})$ satisfy it. This approach was used in \cite{MY IJM}, but it does not give explicit formulas. In this paper we present an explicit formula for an almost projection $P$ when $a,b$ satisfy the relations (\ref{K2}).

Note that $P$ is unitarily equivalent to 
$$
P'=P'(a,b)=\left(\begin{matrix}
1+(1-a)^{1/2}(a-b)(1-a)^{1/2}&(1-a)^{1/2}(b-a)a^{1/2}\\
a^{1/2}(b-a)(1-a)^{1/2}&a^{1/2}(a-b)a^{1/2}
\end{matrix}\right)
$$ 
via the unitary $U=\left(\begin{matrix}(1-a)^{1/2}&-a^{1/2}\\a^{1/2}&(1-a)^{1/2}
\end{matrix}\right)$, $P'=U^*PU$.    

Set 
$$
P''=P''(a,b)=\left(\begin{matrix}
1+(1-a)(a-b)(1-a)&(1-a)(b-a)a\\
a(b-a)(1-a)&a(a-b)a
\end{matrix}\right).
$$ 

\begin{lem}\label{L3}
Let $a$, $b$ be selfadjoints satisfying $0\leq a,b\leq 1$ and $p(a)=p(b)$ for $p(t)=t(1-t)$ and $p(t)=t^2(1-t)$. Then $P'(a,b)=P''(a,b)$.

\end{lem}

\begin{proof}

Suppose that $a-a^2=b-b^2$ and $a^2-a^3=b^2-b^3$. Then
$$
(a^2-a)(a-b)=a^3-a^2-(a^2-a)b=a^3-a^2-(b^2-b)b=0.
$$
Similarly, $(b^2-b)(a-b)=0$.

There is (see \cite{Man-K}) a universal $C^*$-algebra $D$ generated by two selfadjoint positive contractions $a$, $b$ subject to the relations $(a^2-a)(a-b)=(b^2-b)(a-b)=0$. It was shown in \cite{Man-K} that $D\subset C([-1,1];M_2)$ is a subalgebra of matrix-valued functions 
$$
f=\left(\begin{matrix}f_{11}&f_{12}\\f_{21}&f_{22}\end{matrix}\right)\in C([-1,1];M_2)
$$ 
such that
$$
f_{11}(-1)=0;\quad f_{12}(t)=f_{21}(t)=f_{22}(t)=0 \ \mbox{for\ }t\in[-1,0];\quad f_{12}(1)=f_{21}(1)=0.
$$   
The generators $a,b\in D$ are given by the formulas
$$
a(t)=\left\lbrace
\begin{array}{cl}\left(\begin{matrix}\cos^2\frac{\pi}{2}t&0\\0&0\end{matrix}\right)&\mbox{for\ }t\in[-1,0];\\
\left(\begin{matrix}1&0\\0&0\end{matrix}\right)&\mbox{for\ }t\in[0,1],
\end{array}\right.
$$
$$
b(t)=\left\lbrace
\begin{array}{cl}\left(\begin{matrix}\cos^2\frac{\pi}{2}t&0\\0&0\end{matrix}\right)&\mbox{for\ }t\in[-1,0];\\
\left(\begin{matrix}\cos^2\frac{\pi}{2}t&\cos\frac{\pi}{2}t\sin\frac{\pi}{2}t\\
\cos\frac{\pi}{2}t\sin\frac{\pi}{2}t&\sin^2\frac{\pi}{2}t\end{matrix}\right)&\mbox{for\ }t\in[0,1].
\end{array}\right.
$$

Notice that $a(t)-b(t)=0$ for $t\leq 0$, and that $a(t)^{1/2}=a(t)$, $(1-a(t))^{1/2}=1-a(t)$ for $t\geq 0$, therefore, $P'$ equals $P''$.

\end{proof}

\begin{lem}\label{L2}
Let $a$, $b$ be two selfadjoints satisfying $p(a)=p(b)$, where $p(t)$ is either $t(1-t)$ or $t^2(1-t)$ (but not necessarily $0\leq a,b\leq 1$). Then $P''$ is a projection.

\end{lem}

\begin{proof}
Direct calculation.
\end{proof}

\begin{remark}
It follows from Lemma \ref{L2}  that if the relations $p(a)=p(b)$ are true only up to some small value then $P''$ is an almost projection.
\end{remark}

The advantage of $P''$ compared with $P'$ is that $P''$ is polynomial in $a$ and $b$, hence easier to use in calculations.

Set $h(t)=\left\lbrace\begin{array}{cl}1,&\mbox{for\ }t>1;\\t,&\mbox{for\ }0\leq t\leq 1;\\0,&\mbox{for\ }t<0.\end{array}\right.$ Then one has $0\leq h(a), h(b)\leq 1$.

\begin{lem}\label{L4}
Let $a$, $b$ be selfadjoints satisfying $p(a)=p(b)$ for $p(t)=t(1-t)$ and $p(t)=t^2(1-t)$. Then $h(a)$ and $h(b)$ also satisfy these relations, and the projections $P''(h(a),h(b))$ and $P''(a,b)$ are homotopic.

\end{lem}

\begin{proof}
The first claim follows from the continuous functional calculus: 
$$
p(h(a))=h(p(a))=h(p(b))=p(h(b)).
$$

Let $h_0(t)=t$, $h_1(t)=h(t)$ and let $h_s$, $s\in[0,1]$, be a (linear) homotopy connecting $h_0$ with $h_1$. Then $P''_s=P''(h_s(a),h_s(b))$ provides the required homotopy.

\end{proof}

\begin{remark}
If the relations in Lemma \ref{L4} are satisfied only up to some small value then the homotopy constructed above lies in the set of almost projections. 

\end{remark}

Thus, if a pair $(\pi^+,\pi^-)$ is $(F,\varepsilon)$-balanced then 
$$
p(\pi^+,\pi^-)=[P''(A_{\pi^+},A_{\pi^-})]-\left[\left(\begin{smallmatrix}1&0\\0&0\end{smallmatrix}\right)\right]
$$ 
determines a class in $K^0(X)$ when $F$ is large and $\varepsilon$ is small.

\begin{remark}
One can write $P''$ as 
\begin{equation}\label{P''another}
P''(a,b)=\left(\begin{matrix}1&0\\0&0\end{matrix}\right)+\left(\begin{matrix}1{-}a\\-a\end{matrix}\right)(a-b)\Bigl(1{-}a,\ \ -a\Bigr).
\end{equation}

\end{remark}

\section{Example}

Let $X$ be a manifold with $\pi_1(X)=\Gamma$, $\widetilde{X}$ its universal covering, $x\in \widetilde{X}$. Set $Y=\Gamma x\cap B_R$, where $B_R\subset\widetilde{X}$ is the ball of radius $R$ centered at $x$. Then the space $l^2(Y)$ of functions on $Y$ is a finitedimensional subspace of $l^2(\Gamma)$. Let $p$ denote the projection in $l^2(\Gamma)$ onto $l^2(Y)$, and let $V$ be an $n$-dimensional complex vector space. Set $H=l^2(Y)\otimes V$.  

Define $\pi(g)$ on $l^2(Y)\otimes V$ by 
$$
\pi(g)=p\lambda(g)|_{l^2(Y)}\otimes\operatorname{id}_V,
$$ 
where $\lambda$ denotes the left regular representation. Let $B_\pm$ be selfadjoint $\operatorname{End}(V)$-valued functions on $\widetilde{X}$, and let $M_{B_\pm}$ denote the operator of multiplication by the function $B_\pm$ on $l^2(\Gamma)\otimes V$. Set 
$$
\pi^\pm(g)=\pi(g)M_{B_\pm}|_{l^2(Y)}\otimes\operatorname{id}_V.
$$

Let us check when the pair $(\pi^+,\pi^-)$ is $(F,\varepsilon)$-balanced. Let $y\in Y$, $\delta_y\in l^2(Y)$ the corresponding delta-function. Then 
$$
(\pi^\pm(gh)-\pi^\pm(g)\pi^\pm(h))\delta_y=\left\lbrace\begin{array}{cl}(1-B_\pm(hy))B_\pm(y)\delta_{ghy}&\mbox{if\ } hy\in Y, ghy\in Y; \\
B_\pm(y)\delta_{ghy}& \mbox{if\ } hy\notin Y, ghy\in Y; \\0&\mbox{otherwise.}\end{array}\right.  
$$

Set 
$$
m_1=\sup\{\|B_+(y)-B_-(y)\|: y\in Y, ghy\in Y, hy\notin Y\};
$$
$$
m_2=\sup\{\|(1-B_+(hy))B_+(y)-(1-B_-(hy))B_-(y)\|: y\in Y, ghy\in Y, hy\in Y\}.
$$
Then 
$$
\|M_{\pi^+}(g,h)-M_{\pi^-}(g,h)\|=\|(\pi^+(gh)-\pi^+(g)\pi^+(h))-(\pi^-(gh)-\pi^-(g)\pi^-(h))\|=\max(m_1,m_2).
$$
A similar estimate involving also $\gamma\in\Gamma$ can be written for $\|\pi^+(\gamma)M_{\pi^+}(g,h)-\pi^-(\gamma)M_{\pi^-}(g,h)\|$.
Now suppose that the $\operatorname{End}(V)$-valued functions $B_\pm$ have small variation, i.e. satisfy the estimate 
\begin{equation}\label{Lip}
\|B_\pm(hy)-B_\pm(y)\|<\delta \mbox{\ \ for\ any\ }y\in Y \mbox{\ and\ any\ }h\in F\subset\Gamma. 
\end{equation}

\begin{lem}
Assume that $B_+(y)=B_-(y)$ for all $y$ with $d(y,x)\geq R$, and that (\ref{Lip}) holds. There exists a constant $C$ such that 
if the pair $(\pi^+,\pi^-)$ is $(F,\varepsilon)$-balanced then $\|p(B_+)-p(B_-)\|<\varepsilon+C\delta$ for $p=t(1-t)$ and $p(t)=t^2(1-t)$, and, conversely, if $\|p(B_+)-p(B_-)\|<\varepsilon$ then the pair $(\pi^+,\pi^-)$ is $(F,\varepsilon+C\delta)$-balanced.

\end{lem}

\begin{proof}
Up to $C\delta$, we may not distinguish between $B_\pm(hy)$ and $B_\pm(y)$. Then the claim becomes obvious.

\end{proof}

Remark that the pairs $(B_+,B_-)$ with the above properties can be considered as elements of $K_0(C_0(\widetilde{X}))=K^0_c(\widetilde{X})$. 

Starting from a class $z\in K_0(C_0(\widetilde{X}))$, take a pair $(B_+,B_-)$ that represents $z$, then construct $(\pi^+_R,\pi^-_R)$, where $R$ is the radius of the ball that determines $Y$, as above. The pair $(\pi^+_R,\pi^-_R)$ is asymptotically balanced as $R\to\infty$. Then, using $\pi^\pm_R$, define $A_{\pm,R}$ as in (\ref{A}), where the cocycle $\gamma$ determines the Mishchenko line bundle. Finally take $P''(A_{+,R},A_{-,R})$, which determines a class in $K^0(X)$.

\begin{lem}
The construction described above defines a map 
$$
z\mapsto [P''(A_{+,R},A_{-,R})]-\left[\left(\begin{smallmatrix}1&0\\0&0\end{smallmatrix}\right)\right], \quad K^0_c(\widetilde{X})\to K^0(X), 
$$
which coinsides with the direct image map.

\end{lem}

\begin{proof}
Let $B_-$ be a constant projection on $\widetilde{X}$, and let $B_+$ be a projection-valued function on $\widetilde{X}$, which is equal to $B_-$ at infinity. 
Let $\pi^-(g)=\lambda(g)\otimes\operatorname{id}_V$ be the left regular representation on $l^2(\Gamma)\otimes V$, and let $\pi^+(g)=\pi^-(g)M_{B_+}$. Then $\pi^+(g)-\pi^-(g)$ is compact for any $g\in\Gamma$, so $(\pi^+,\pi^-)$ is a Fredholm representation. Define $A_\pm$ by $A_\pm(x)=\pi^\pm(p(x))$, where $p$ is defined in (\ref{proj}), $x\in X$. Then $A_\pm$ are projection-valued functions on $X$, and $[P''(A_+,A_-)]-\left[\left(\begin{smallmatrix}1&0\\0&0\end{smallmatrix}\right)\right]$ is unitarily equivalent to $[A_+]-[A_-]$. 

It is shown in \cite{Mish-Tel} that $i_!([B_+]-[B_-])=[A_+]-[A_-]$, where $i_!:K^0_c(\widetilde{X})\to K^0(X)$ is the direct image map. Our proof would follow if we show that $P''(A_{+,R},A_{-,R})$ is convergent to $P''(A_+,A_-)$ in norm, as $R\to\infty$ ($*$-strong convergence is obvious).

Denote by $L_R$ the orthogonal complement to $l^2(Y)\otimes V$ in $l^2(\Gamma)\otimes V$. As the word length metric on $\Gamma$ is quasi-equivalent to the metric on $\widetilde{X}$, so there is a constant $0<C<1$ such that $A_\pm\xi, A_{\pm,R}\xi\in L_{CR}$ when $\xi\in L_R$, for sufficiently great $R$. Therefore, for the restriction onto $L_R$ we have, using (\ref{P''another}), the following estimate: 
$$
\left\|\left(P''(A_{+,R},A_{-,R})-\left(\begin{smallmatrix}1&0\\0&0\end{smallmatrix}\right)\right)\left|_{L_R}\right.\right\|\leq 
\|(A_{+,R}-A_{-,R})|_{L_{CR}}\|\leq\sup_{x\notin B_{CR}}|B_+(x)-B_-(x)|\to 0 \mbox{\ as\ }R\to\infty,
$$
and, similarly, 
$$
\lim_{R\to\infty}\left\|\left(P''(A_{+,R},A_{-,R})-\left(\begin{smallmatrix}1&0\\0&0\end{smallmatrix}\right)\right)\left|_{L_R}\right.\right\|=0.
$$
As all the operators involved are selfadjoint, so this, together with the $*$-strong convergence, implies the norm convergence.

\end{proof}

{\bf Acknowledgement.} The author is grateful to the Vietnam Institute of Advanced Studies in Mathematics for hospitality.

\end{document}